\documentclass{article}
\usepackage[margin=1.5in]{geometry}
\usepackage{graphicx} % Required for inserting images
\usepackage{amsthm}
\usepackage{amssymb}
\usepackage{amsmath} 
\usepackage{mathtools}
\usepackage{xcolor}
\usepackage{enumerate}
\usepackage{microtype}
\usepackage{hyperref}
\usepackage[natbib,style=numeric,sorting=nyt]{biblatex}
\newcommand{\bin}[2]{\binom{#1}{#2}}
\newcommand{\tow}{\operatorname{tow}}

\newtheorem{theorem}{Theorem}

\newtheorem{lemma}[theorem]{Lemma}

\if10     % 11 = on,  10 = off
\usepackage[mathlines]{lineno}
\newcommand*\patchAmsMathEnvironmentForLineno[1]{%
  \expandafter\let\csname old#1\expandafter\endcsname\csname #1\endcsname
  \expandafter\let\csname oldend#1\expandafter\endcsname\csname end#1\endcsname
  \renewenvironment{#1}%
     {\linenomath\csname old#1\endcsname}%
     {\csname oldend#1\endcsname\endlinenomath}}% 
\newcommand*\patchBothAmsMathEnvironmentsForLineno[1]{%
  \patchAmsMathEnvironmentForLineno{#1}%
  \patchAmsMathEnvironmentForLineno{#1*}}%
\AtBeginDocument{%
\patchBothAmsMathEnvironmentsForLineno{equation}%
\patchBothAmsMathEnvironmentsForLineno{align}%
\patchBothAmsMathEnvironmentsForLineno{flalign}%
\patchBothAmsMathEnvironmentsForLineno{alignat}%
\patchBothAmsMathEnvironmentsForLineno{gather}%
\patchBothAmsMathEnvironmentsForLineno{multline}%
}
\linenumbers
\fi
\def\inst#1{$^{#1}$}

\title{On edge-ordered Ramsey numbers of hypergraphs\thanks{Both authors were supported by grant no.\ 23-04949X of the Czech Science Foundation (GA\v{C}R).
J. Glišić was also supported by the Charles University Grant Agency (GAUK) project no.\ 450526.
}}
\author{
Martin Balko\inst{1}
\and 
Jelena Gli\v{s}i\'{c}\inst{1}
}

\begin{document}

\maketitle

\begin{center}
{\footnotesize
\inst{1} 
Department of Applied Mathematics, \\
Faculty of Mathematics and Physics, Charles University, Czech Republic \\
\texttt{\{balko,glisic\}@kam.mff.cuni.cz}
}
\end{center}

\begin{abstract}
For an integer $h\geq2$, an \emph{edge-ordered $h$-uniform hypergraph} $H^<$ is an $h$-uniform hypergraph with a linear ordering of its hyperedges.
For an integer $q \geq 1$, the \emph{edge-ordered Ramsey number} $R_{\mathrm{edge}}(H^<;q)$ of $H^<$ is the smallest positive integer $N$ such that there exists an edge-ordered $h$-uniform hypergraph $G^<$ on $N$ vertices such that, for every $q$-coloring of the hyperedges of $G^<$, there is a monochromatic copy of $H^<$ in $G^<$ preserving the hyperedge order. 
In this note, we show that, for every edge-ordered $h$-uniform hypergraph $F^<$ on $n$ vertices with $m\geq1$ hyperedges, $R_{\mathrm{edge}}(F^<;q)
\leq R_h(F;q)^{Cm^2}$ where $F$ is the underlying unordered hypergraph, $R_h(F;q)$ is its standard $q$-color Ramsey number, and $C=C(h,q)>0$ is a constant.
For fixed $h\geq3$ and $q\geq2$, this shows that prescribing the hyperedge order does not require an increase in the tower height of the classical upper bound.
It also settles a problem posed by Fox and Li (2020), and for $h=2$ gives the bound $2^{O(n^4\log n)}$, improving their bounds for every fixed $q \geq 3$.
\end{abstract}

\section{Introduction}

Let $h\geq2$ and $q\geq1$ be integers and $H$ an $h$-uniform hypergraph.
The \emph{Ramsey number} $R_h(H;q)$ is the minimum positive integer $N$ such that every $q$-coloring of the edges of the complete $h$-uniform hypergraph $K^{(h)}_N$ on $N$ vertices contains a monochromatic copy of~$H$.
We write $R_h(n;q)=R_h(K_n^{(h)};q)$.

The growth rate of Ramsey numbers is a classical and extensively studied topic in combinatorics, with many applications in various fields of mathematics.
For fixed $h\geq3$ and $q\geq2$, classical bounds~\cite{erdRado52,ramseyBook} give $\tow_{h-1}(\Omega(n^2))\leq R_h(K_n^{(h)};q)\leq\tow_h(O(n))$, with the lower bound strengthened to $\tow_h(\Omega(n))$ for $q\geq4$, where the \emph{tower function} $\tow_k(\cdot)$ of height $k$ is defined as $\tow_1(x)=x$ and $\tow_{k+1}(x)=2^{\tow_k(x)}$.

Natural refinements, motivated by problems from discrete geometry, arise when the hypergraphs are equipped with linear orders of vertices~\cite{bckk20,cfls17} or hyperedges~\cite{balViz20} and embeddings are required to preserve these orders in addition to the incidence structure.
For vertex-ordered graphs, this requirement can substantially change the growth of Ramsey numbers, as even ordered matchings may have superpolynomial Ramsey numbers~\cite{bckk20,cfls17}.
The study of vertex-ordered Ramsey numbers has developed into an active area of research, with a substantial body of results concerning both general bounds and specific classes of ordered graphs; see the recent survey~\cite{balkoSurvey}.

Graphs with linearly ordered edges, called \emph{edge-ordered graphs}, have attracted increasing attention in extremal combinatorics, with a growing body of results concerning their Tur\'an numbers~\cite{gerbnerEtAlTuran,kucheriyaTardosAlmostLinear,tardos19}.
Balko and Vizer~\cite{balViz20} initiated the study of Ramsey numbers of edge-ordered graphs and proved their finiteness using the Graham--Rothschild theorem.
More generally, for an edge-ordered $h$-uniform hypergraph $H^<$, the \emph{edge-ordered Ramsey number} $R_{\mathrm{edge}}(H^<;q)$ is the minimum number of vertices in an edge-ordered $h$-uniform hypergraph such that every $q$-coloring of its hyperedges contains a monochromatic copy of $H^<$ preserving the hyperedge order.
Note that the edge order of the host is chosen as part of the construction, before the coloring.

The edge-ordered Ramsey theory remains a relatively young area, with many basic questions concerning the growth of its Ramsey numbers still open.
Fox and Li~\cite{foxLi20} proved the bound 
\begin{equation}
\label{eq-FoxLi}
R_{\mathrm{edge}}(G^<;q)\in 2^{O(n^{2q-2}(\log n)^q)}
\end{equation}
for every $n$-vertex edge-ordered graph $G^<$ and every fixed $q\geq2$, as well as polynomial bounds in the two-color case for graphs of bounded degeneracy.
For higher uniformities, the finiteness of edge-ordered Ramsey numbers follows from a general structural Ramsey theorem of Hubi\v{c}ka and Ne\v{s}et\v{r}il~\cite[Section~4.3.4]{hubNes19}, which allows linear orders on both the vertices and the tuples in each relation of a finite relational structure.
Their proof is constructive, but the resulting upper bounds are enormous, as noted by Fox and Li~\cite[Section~8.4]{foxLi20}, who explicitly raised the problem of obtaining better quantitative estimates for edge-ordered hypergraph Ramsey numbers, already in the two-color case.

A natural quantitative refinement of this problem is to ask whether, for fixed uniformity $h$ and number of colors $q$, the edge-ordered Ramsey numbers of $n$-vertex hypergraphs can be bounded by a tower of exponentials whose height is independent of $n$.
More precisely, the classical upper bound $R_h(n;q)\leq\tow_h(O(n))$ provides a natural benchmark.
Can an arbitrary prescribed order on the hyperedges be accommodated without increasing the height of this tower?

\section{Our results}

In this note, we address the problem posed by Fox and Li~\cite{foxLi20}, and we give a better
estimate for hypergraph edge-ordered Ramsey numbers.
Our main result gives an explicit upper bound for the edge-ordered Ramsey number of an arbitrarily edge-ordered $h$-uniform hypergraph $F^<$ on $n$ vertices with $m$ hyperedges in terms of the classical Ramsey number $R_h(F;q)$ of its underlying unordered hypergraph~$F$.

\begin{theorem}
\label{thm-main}
For all integers $h\geq2$ and $q \geq 1$, there is a
number $C=C(h,q)>0$ such that, for every positive integer $n$ and every
edge-ordered $h$-uniform hypergraph $F^<$ on $n$ vertices with $m\geq1$ hyperedges,
\[
 R_{\mathrm{edge}}(F^<;q)\leq \left(2^{m^2\ln(m+1)} \cdot R_h(F;q)^{mn}\right)^C.
\]
We also have $R_{\mathrm{edge}}(F^<;q)\leq R_h(F;q)^{C m^2}$.
In particular, there is a number $A=A(h,q)$ such that
\[
 R_{\mathrm{edge}}(F^<;q)\leq
 \begin{cases}
  2^{A n^4\ln{n}},& h=2,\\
  \tow_h(A n),& h\geq3.
 \end{cases}
\]
\end{theorem}

For $h=2$, our bound $2^{A n^4\ln n}$ improves the multicolor bound~\eqref{eq-FoxLi} of Fox and Li~\cite{foxLi20} for every fixed $q\geq3$, whereas their two-color bound $2^{O(n^2(\log n)^2)}$ is stronger.
In particular, for fixed $h\geq3$ and $q\geq2$, Theorem~\ref{thm-main} yields the bound $\tow_h(O(n))$, showing that prescribing the hyperedge order does not require an increase in the tower height of the classical upper bound.
For a complete $F^<$, we have $R_{\mathrm{edge}}(F^<;q)\geq R_h(n;q)\geq n$, so, taking $C\geq1$, Theorem~\ref{thm-main} also gives, for fixed $h,q\geq2$,
\[
 \begin{aligned}
 0&\leq\ln\ln R_{\mathrm{edge}}(F^<;q)-\ln\ln R_h(n;q)\\
 &\leq\ln C+2\ln m\leq\ln C+2h\ln n.
 \end{aligned}
\]
The relative error therefore tends to zero whenever $\ln\ln R_h(n;q)=\omega(\log n)$.
Together with the classical lower bounds above and the bound $R_3(n;3)\in 2^{n^{\Omega(\log n})}$ of Conlon, Fox, and Sudakov~\cite{cfs10}, this implies
\[
 \ln\ln R_{\mathrm{edge}}(F^<;q)=(1+o(1))\ln\ln R_h(n;q)
\]
for fixed $h\geq4$ and $q\geq2$, or $h=3$ and $q\geq3$, uniformly over all edge orders of $F^<$.
Thus, in these cases the comparison with the classical Ramsey number is asymptotically exact on the double-logarithmic scale.

The proof of Theorem~\ref{thm-main} is based on the hypergraph container method~\cite{ST} and is inspired by the work of Conlon, Dellamonica, La Fleur, R\"{o}dl, and Schacht~\cite{inducedRamsey}, who used this method to bound induced hypergraph Ramsey numbers in terms of their classical counterparts.
We assign labels to host hyperedges and encode copies with labels increasing in the prescribed edge order as hyperedges of an auxiliary hypergraph.
The container theorem controls color classes avoiding these copies, and a random choice of one label for each host hyperedge produces, with positive probability, an edge order for which such classes cannot cover the host.

\section{Preliminaries}

For a positive integer $n$, we use $[n]$ to denote the set $\{1,\dots,n\}$.
We use $\ln$ for the natural logarithm and $\log$ for the logarithm to base~$2$. 
We omit ceiling and floor signs whenever they are not crucial.

For an $m$-uniform hypergraph $\mathcal H$ on $v$ vertices with average degree $d>0$, we let the \emph{codegree} $d_{\mathcal H}(S)$ of $S$ be the number of hyperedges of $\mathcal H$ containing $S$.
For a vertex $x$ of $\mathcal{H}$, $\tau \in (0,1/2)$, and each integer $j$ with $2 \leq j \leq m$, we write
\[
 d^{(j)}(x)=\max\{d_{\mathcal H}(S):x\in S\subseteq V(\mathcal H),\ |S|=j\},
\quad
 \delta_j=\frac{1}{\tau^{j-1}vd}\sum_{x\in V(\mathcal H)}d^{(j)}(x),
\]
and
\[
 \delta(\mathcal H,\tau)
 =2^{\bin{m}{2}-1}\sum_{j=2}^m2^{-\bin{j-1}{2}}\delta_j.
\]

The proof of Theorem~\ref{thm-main} is based on the following result by Saxton and Thomason~\cite{ST}, called the \emph{Container theorem}; see Corollary 3.6 in~\cite{ST} and the paragraph immediately following it.

\begin{theorem}[The Container theorem~\cite{ST}]
\label{thm-container}
For an $m$-uniform hypergraph $\mathcal H$ on $v$ vertices with average degree $d>0$, if $0<\eta,\tau<1/2$ and
$\delta(\mathcal H,\tau)\leq\eta/(12m!)$, then there is a family $\mathcal C$ of
subsets of $V(\mathcal H)$ satisfying the following three conditions:
\begin{enumerate}[(a)]
\item\label{item-container1} Every independent set in $\mathcal H$ is contained in some $C\in\mathcal C$.
\item\label{item-container2} $e(\mathcal H[C])\leq\eta e(\mathcal H)$ for every $C\in\mathcal C$.
\item\label{item-container3} $\ln|\mathcal C|\leq c_m\ln(1/\eta)v\tau\ln(1/\tau)$,
where $c_m=800(m!)^3m$ is permissible.
\end{enumerate}
\end{theorem}

Intuitively, the container theorem gives a compact description of all independent sets in the hypergraph $\mathcal{H}$. 
Every independent set lies inside one of relatively few \emph{containers} $C \in \mathcal{C}$, each spanning only a small portion of the hyperedges of $\mathcal{H}$.
The codegree condition $\delta(\mathcal H,\tau)\leq\eta/(12m!)$ ensures that hyperedges of $\mathcal{H}$ do not overlap too heavily on small vertex sets. 
The parameter $\eta$ controls the portion of hyperedges in each container, while $\tau$ controls the number of containers. 

In typical applications, one constructs a hypergraph whose hyperedges represent forbidden configurations, so that independent sets correspond to the objects under consideration. 
A supersaturation argument then converts the bound on the number of hyperedges in each container into
a bound on its size. 
One can subsequently count subsets of the containers
or use a union bound over the container family in a probabilistic argument.

\section{Proof of Theorem~\ref{thm-main}}

For integers $q \geq 2$, $h \geq 2$, and $n \geq h+1$, let $F^<$ be an edge-ordered $h$-uniform hypergraph with the vertex set $[n]$ and assume first that it has no isolated vertices.
We use $m\geq2$ to denote the number of hyperedges of~$F^<$, which are ordered as $f_1<\cdots<f_m$.
Let $a=n!/|\operatorname{Aut}(F^<)|$, where $\operatorname{Aut}(F^<)$ consists of the vertex permutations that preserve every hyperedge $f_i$ as a set, and put $\rho=\binom{n}{h}/m$.
We set $R = R_h(F;q)$ and $T=q(m-1)a\binom{R}{n}+1$.
The choice of $T$ will force one monochromatic ordered edge tuple to recur at $m$ labels in any $q$-coloring of the labeled hyperedges on an $R$-set.
We assume that $N$ is an integer satisfying
\[N \geq \left(4q c_mT^2\binom{R}{h}\ln{\left(4qa\binom{T}{m}\binom{R}{n}\right)}\right)^{4m},\]
where $c_m$ is the parameter from Theorem~\ref{thm-container}.
This lower bound on $N$ is chosen to ensure both the codegree condition and the final union bound.
Note that $h+1\leq n\leq hm$, $\rho\geq1$, $R\geq n$, and $T\geq q(m-1)+1\geq m$.
Since $\binom{R}{h}\geq R$, $c_m>1$, and $\ln(4qa\binom{T}{m}\binom{R}{n})>1$, this choice also gives $N\geq R$.

We now aim to define the edge-ordered complete $h$-uniform hypergraph $K^<$ on $[N]$ such that every $q$-coloring of $K^<$ contains a monochromatic copy of $F^<$.
Let $\lambda \colon \binom{[N]}{h} \to [T]$ be an arbitrary assignment of labels from $[T]$ to hyperedges on $[N]$.
For each such $\lambda$, we define the complete edge-ordered $h$-uniform hypergraph $K^<(\lambda)$ on $[N]$ such that two hyperedges $e,f \in \binom{[N]}{h}$ satisfy $e<f$ if and only if $\lambda(e)<\lambda(f)$, or $\lambda(e)=\lambda(f)$ and $e$ is lexicographically smaller than $f$.
That is, we order the hyperedges from $\binom{[N]}{h}$ according to their labels selected by $\lambda$ while breaking ties using the lexicographical order.

In the proof, we show that there is a choice of $\lambda$ such that $K^<(\lambda)$ is the desired edge-ordering $K^<$ of $K^{(h)}_N$.
To do so, we convert the edge-order constraint into a hypergraph containment condition and then apply the Container theorem (Theorem~\ref{thm-container}).

We now define a suitable $m$-uniform hypergraph $\mathcal{H}$.
We first let $\Omega =\binom{[N]}{h}\times[T]$ be the set of hyperedges on $[N]$ that are labeled by numbers from $[T]$ that we eventually use to sort selected hyperedges and obtain the desired copy of $F^<$.
For each set $U$ of $n$ vertices from $[N]$, consider an arbitrary one-to-one correspondence $\phi:[n]\longrightarrow U$.
We extend this mapping to hyperedges by defining the mapping $\phi^{(h)} \colon \binom{[n]}{h} \to \binom{U}{h}$ given by $\phi^{(h)}(f) = \{\phi(v) \colon v \in f\}$.
For each set $L$ of $m$ labels $\ell_1<\cdots<\ell_m$ from $[T]$, we define the set
\[
 A(\phi,L)=\{(\phi^{(h)}(f_i),\ell_i):1\leq i\leq m\}\subseteq\Omega,
\]
which corresponds to a copy of $F^<$ on $U$ if we order the hyperedges $\phi^{(h)}(f_i)$ by the increasing labels~$\ell_i$.
Let $\mathcal H$ be the $m$-uniform hypergraph on $\Omega$ whose edges are all distinct sets $A(\phi,L)$.
The vertices of $\mathcal{H}$ correspond to possible labeled host hyperedges, while its hyperedges are possible copies of~$F^<$.

We first verify the codegree condition from Theorem~\ref{thm-container} for a suitable choice of the parameters $\eta$ and $\tau$ so that we can apply this result to our constructed $m$-uniform hypergraph~$\mathcal{H}$.

\begin{lemma}
\label{lem-codegree}
For $\tau = N^{-1/(2m)}$ and every $\eta \in (0,1/2)$, if 
\begin{equation}
\label{eq-lemCodegree}
N \geq \left(\frac{48 \cdot 2^{\binom{m}{2}}m\rho Tn m!}{\eta}\right)^2,
\end{equation}
then the constructed $m$-uniform hypergraph $\mathcal{H}$ satisfies the assumptions of Theorem~\ref{thm-container}.
\end{lemma}
\begin{proof}
Since $m\geq2$, $n\geq3$, and $\rho\geq1$, condition~\eqref{eq-lemCodegree} gives $N\geq(48n)^2\geq2n$ and $N\geq48^2 2^{m(m-1)}\geq2^{2m+1}$.
First, the number $v$ of vertices of $\mathcal{H}$ is $|\Omega| = T \binom{N}{h}$.
Moreover, the set of underlying vertices used by $A(\phi,L)$ is exactly $U$, because $F^<$ has no isolated vertices.
The set of labels used in $A(\phi,L)$ is exactly $L$, so both $U$ and $L$ are determined by this hyperedge of $\mathcal{H}$.
For fixed $(U,L)$, two bijections $\phi,\psi:[n]\longrightarrow U$ give the same hyperedge of $\mathcal{H}$ if and only if $\psi^{-1}\phi\in\operatorname{Aut}(F^<)$, so there are exactly $a$ distinct edges for each pair $(U,L)$.
Thus, the number of hyperedges of $\mathcal{H}$ is $a\binom{N}{n}\binom{T}{m}$.
Altogether, it follows that the average degree $d$ of $\mathcal H$ satisfies
\begin{equation}\label{eq-average}
 d=\frac{ma \binom{N}{n}\binom{T}{m}}{T\bin{N}{h}}
 =\frac {a\binom{T}{m}}{\rho T}\binom{N-h}{n-h},
\end{equation}
where we used $\rho=\bin{n}{h}/m$ and $\binom{N}{n}\binom{n}{h} = \binom{N}{h}\binom{N-h}{n-h}$.

Let $S\subseteq\Omega$ have size $j\geq2$. 
If $S$ contains two elements of $\Omega$ with the same underlying $h$-set, or with the same label, then it is not contained in a hyperedge of $\mathcal H$ and $d_\mathcal{H}(S)=0$. 
Otherwise, the underlying $h$-sets of $S$ are distinct, and their union has size $w\geq h+1$. 
If $w>n$, then $d_\mathcal{H}(S)$ is again zero.
If $w\leq n$, each hyperedge of $\mathcal H$ containing $S$ has its ground $n$-set containing that union, and there are at most $\binom{T}{m}$ possible label sets and $a$ edges for each such pair.
Thus, 
\[
 d_{\mathcal H}(S)\leq a\binom{T}{m}\binom{N-w}{n-w}
 \leq a\binom{T}{m}\binom{N-h-1}{n-h-1}.
\]
Together with \eqref{eq-average} and $N\geq2n$, this implies, for every vertex $x$ of $\mathcal{H}$ and every integer $j$ with $2\leq j\leq m$,
\begin{equation}\label{eq-codegree}
 \frac{d^{(j)}(x)}d
 \leq \rho T\frac{n-h}{N-h}\leq\frac{2\rho Tn}{N}.
\end{equation}
Since $N\geq2^{2m+1}$, the choice of $\tau$ implies $0<\tau<1/2$. 
By the definition of $\delta_j$ and \eqref{eq-codegree},
\[
 \delta_j\leq\frac{2\rho Tn}{N\tau^{j-1}}  \leq 2\rho TnN^{-1+(j-1)/(2m)}\leq 2\rho TnN^{-1/2}.
\]
It then follows from~\eqref{eq-lemCodegree} that
\[
 \delta(\mathcal H,\tau)
 \leq 2^{\bin{m}{2}}m2\rho TnN^{-1/2} \leq\frac{\eta}{24m!}<\frac{\eta}{12m!},
\]
and the assumptions of Theorem~\ref{thm-container} are satisfied.
\end{proof}

To ensure that the containers contain too few edges of $\mathcal H$ to account for copies on more than half of the $R$-sets, we now set
\begin{equation}
\label{eq-eta}
\eta = \left(4qa\binom{T}{m}\binom{R}{n}\right)^{-1}.
\end{equation} 
Then $0<\eta\leq1/(4q)\leq1/8$, so $\ln(1/\eta)>1$.
Observe that, by the choice of $T$,
\[
 \frac1\eta=\frac{4(T-1)}{m-1}\binom{T}{m}\leq\frac{4T^{m+1}}{m-1},
\]
and $m\rho=\binom{n}{h}$ and $n\leq\binom{n}{h}\leq\binom{R}{h}$ give $m\rho n\leq\binom{R}{h}^2$.
Also, $m\geq2$ and $c_m>2^m m \cdot m!$ imply $(4qc_m)^{2m}\geq192\cdot2^{\binom{m}{2}}m!/(m-1)$.
Consequently, our choice of $N$ satisfies~\eqref{eq-lemCodegree}, since
\begin{align*}
 \frac{48\cdot2^{\binom{m}{2}}m\rho Tn m!}{\eta}
 &\leq\frac{192\cdot2^{\binom{m}{2}}m!}{m-1}\,T^{m+2}\binom{R}{h}^2\\
 &\leq(4qc_m)^{2m}T^{4m}\binom{R}{h}^{2m}\\
 &\leq\left(4q c_mT^2\binom{R}{h}\ln(1/\eta)\right)^{2m}\leq\sqrt{N}.
\end{align*}
It now follows from Lemma~\ref{lem-codegree} and Theorem~\ref{thm-container} applied with $\eta$ that there is a family $\mathcal C$ of subsets of $V(\mathcal H)$ satisfying the three properties stated in Theorem~\ref{thm-container}.
We first show that every union of $q$ containers from $\mathcal{C}$ misses many elements of $\Omega$.

\begin{lemma}
\label{lem-missing}
For any $C_1,\ldots,C_q\in\mathcal C$, we have
\begin{equation}\label{eq-missing}
 \left|\Omega\setminus\left(\bigcup_{i=1}^qC_i\right)\right|\geq\frac{\binom{N}{h}}{2\binom{R}{h}}.
\end{equation}
\end{lemma}
\begin{proof}
Let $D=\Omega\setminus\left(\bigcup_{i=1}^qC_i\right)$ and suppose for contradiction that $|D|<\binom{N}{h}/(2\binom{R}{h})$. 
A uniformly chosen $R$-set $W\subseteq[N]$ contains the
underlying $h$-set of any fixed missing element of $\Omega$ with probability
\[
 \frac{\bin{N-h}{R-h}}{\bin{N}{R}}
 =\frac{\bin{R}{h}}{\bin{N}{h}}.
\]
The union bound with $|D|<\binom{N}{h}/(2\binom{R}{h})$ therefore implies that more than half of the $R$-sets $W$ have every element of $\bin{W}{h}\times[T]$ in $\bigcup_i C_i$.
We call such $R$-sets \emph{good}. 
On each good $W$, color each element of $\bin{W}{h}\times[T]$ with the least
index $i$ for which it lies in $C_i$. 
Since $W$ is good, this is a valid coloring $\chi$ of $\bin{W}{h}\times[T]$ with $q$ colors.

We now show that there are an injection $\phi:[n]\longrightarrow W$ and $L \in\binom{[T]}{m}$ such that $A(\phi,L)$ has all elements of the same color in $\chi$.
Since $|W| = R=R_h(F;q)$, for each label $\ell\in[T]$, the $q$-coloring $e\mapsto\chi(e,\ell)$ of $h$-sets in $W$ has a copy of $F$ given by an injection $\phi_\ell:[n]\longrightarrow W$, with all its hyperedges colored by some color $c_\ell \in [q]$.
Choose one such injection and color for each label.
There are at most $qa\binom Rn$ possible pairs $((\phi_\ell^{(h)}(f_1),\ldots,\phi_\ell^{(h)}(f_m)),c_\ell)$.
Thus, since $T=q(m-1)a\binom{R}{n}+1$, the pigeonhole principle implies that some pair $((\phi^{(h)}(f_1),\ldots,\phi^{(h)}(f_m)),c)$ occurs at least $m$ times.
Let its labels be $\ell_1<\cdots<\ell_m$.
Then every hyperedge $\phi^{(h)}(f_i)$ has color $c$ at each of those labels.
In particular, $A(\phi,\{\ell_1,\ldots,\ell_m\})$ has all elements of the same color $c \in [q]$ in $\chi$.

Thus, for each good $W$, we obtain an edge $A(\phi,L)$ of $\mathcal H$ contained in one $C_i$ and supported on $W$.
We now count incidences between good $R$-sets and such container edges.
Each edge of $\mathcal H$ has exactly $n$ ground vertices and thus belongs to exactly $\bin{N-n}{R-n}$ ground $R$-sets. 
Since there are more than $\frac{1}{2}\binom{N}{R}$ good $R$-sets, it follows that
\[
 \sum_{i=1}^q e(\mathcal H[C_i])
 >\frac{\frac{1}{2}\binom{N}{R}}{\binom{N-n}{R-n}}
 =\frac{\binom{N}{n}}{2\binom{R}{n}}.
\]
This lower bound remains valid even if container edges overlap, as the sum only counts them additional times. 
On the other hand, Property~\eqref{item-container2} of Theorem~\ref{thm-container}, $e(\mathcal H)=a\binom{N}{n}\binom{T}{m}$, and the choice~\eqref{eq-eta} of $\eta$ give
\[
 \sum_{i=1}^q e(\mathcal H[C_i]) \leq q\eta e(\mathcal{H})
= q\eta a\binom{N}{n}\binom{T}{m} = \frac{\binom{N}{n}}{4\binom{R}{n}},
\]
a contradiction.
\end{proof}

Finally, we show that at least one edge-ordering $K^<(\lambda)$ gives the desired result.

\begin{theorem}
\label{thm-explicit}
With $F^<,N,T$ as above, there is a labeling $\lambda \colon \binom{[N]}{h} \to [T]$ such that every $q$-coloring of $K^<(\lambda)$ contains a monochromatic copy of $F^<$ preserving its edge order.
\end{theorem}

\begin{proof}[Proof of Theorem~\ref{thm-explicit}]
To find a suitable labeling $\lambda \colon \binom{[N]}{h} \to [T]$, we use the probabilistic method.
For every $e\in\bin{[N]}h$, we choose $\lambda(e)$ independently uniformly at random from $[T]$, and we let
\[
 X=\left\{(e,\lambda(e)):e\in\bin{[N]}h\right\}\subseteq\Omega.
\]
Fix $C_1,\ldots,C_q\in\mathcal C$.
Let $d_e$ be the number of labels $\ell \in [T]$ with $(e,\ell)\notin \bigcup_iC_i$.
By Lemma~\ref{lem-missing}, $\sum_{e \in \binom{[N]}{h}} d_e \geq \frac{\binom{N}{h}}{2\binom{R}{h}}$.
Thus, by the choice of $\lambda$, we then get
\[
 \Pr\left[X\subseteq \bigcup_iC_i\right]=\prod_{e\in\binom{[N]}h}\left(1-\frac{d_e}{T}\right)
 \leq\exp\left(-\frac1T\sum_{e \in \binom{[N]}{h}} d_e\right)
 \leq\exp\left(-\frac{\binom{N}{h}}{2T\binom{R}{h}}\right).
\]
The first inequality also holds if some factor is zero.

Recall that $\tau = N^{-1/(2m)}$.
Thus, $\ln{(1/\tau)} \leq \sqrt{1/\tau}$ and we obtain from the choice of $N$ that
\[
 \tau\ln(1/\tau)\leq N^{-1/(4m)}\leq \left(4q c_mT^2\binom{R}{h}\ln(1/\eta)\right)^{-1}.
\]
Consequently, by Property~\eqref{item-container3} of Theorem~\ref{thm-container} and by $v=T\binom{N}{h}$,
\[
 q\ln|\mathcal C|
 \leq q c_m\ln(1/\eta)T\binom{N}{h}\tau\ln(1/\tau)
 \leq\frac{\binom{N}{h}}{4T\binom{R}{h}}.
\]
It thus follows from the union bound that
\begin{align*}
 \Pr\left[X\subseteq\bigcup_{i=1}^q C_i
 \text{ for some }C_1,\ldots,C_q\in\mathcal C\right]
 \leq |\mathcal C|^q\exp\left(-\frac{\binom{N}{h}}{2T\binom{R}{h}}\right)
 \leq\exp\left(-\frac{\binom{N}{h}}{4T\binom{R}{h}}\right)<1.
\end{align*}
Therefore, we can fix $\lambda$ such that the corresponding set $X$ is not contained in any such union.

Consider the edge-ordering $K^< = K^<(\lambda)$.
In every $q$-coloring of $K^<$, we regard the color classes as subsets $X_1,\ldots,X_q$ of $X$. 
If every $X_i$ were
independent in $\mathcal H$, each would lie in a container $C_i$ by Property~\eqref{item-container1} of Theorem~\ref{thm-container}, giving
$X\subseteq\bigcup_iC_i$, contrary to the choice of $X$. Thus some $X_i$
contains an edge $A(\phi,L)$ of $\mathcal H$. Its labels are pairwise distinct
and increase in exactly the prescribed edge order of $F^<$. Therefore
$\phi$ gives a monochromatic edge-order-preserving
copy of $F^<$ in $K^<$.
\end{proof}

To obtain Theorem~\ref{thm-main}, it remains to estimate the size of $N$ by straightforward calculations.

\begin{proof}[Proof of Theorem~\ref{thm-main}]
If $q=1$ or $m=1$, then $R_{\mathrm{edge}}(F^<;q)=n$, and the claimed bounds follow by taking $C$ and $A$ sufficiently large.
We may therefore assume that $q,m\geq2$.
First assume that $F^<$ has no isolated vertices, so $h+1\leq n\leq hm$.
All implicit constants below depend only on $h$ and $q$.
Let $R,T$, and $\eta$ be as above, and set
\[
 B=4q c_mT^2\binom{R}{h}\ln(1/\eta)
\]
and $N=\left\lceil B^{4m}\right\rceil$.
By Theorem~\ref{thm-explicit}, we have $R_{\mathrm{edge}}(F^<;q)\leq N$, and it remains to estimate $N$.

Since $R\geq n$ and $a\binom{R}{n}\leq n!\binom{R}{n}\leq R^n$, we have
\[
 T=q(m-1)a\binom{R}{n}+1
 \leq qma\binom{R}{n}
 \leq qmR^n,
\]
and hence $\ln T \in O(n\log R)$, as $m\leq\binom{n}{h}\leq n^h$.
Moreover, $c_m=800(m!)^3m$ gives $\ln c_m \in O(m\log m)$, while
\[
 \ln(1/\eta)
 =\ln\left(4qa\binom{T}{m}\binom{R}{n}\right)
 \leq\ln(4q)+m\ln T+n\ln R
 \in O(mn\log R).
\]
In particular,
$\ln\ln(1/\eta)=O(\log n+\log\log R)$.
Using also $\binom{R}{h}\leq R^h$, we obtain
\begin{align*}
 \ln B
 &=\ln(4q)+\ln c_m+2\ln T
   +\ln\binom{R}{h}+\ln\ln(1/\eta)\\
 &\in O(m\log m+n\log R).
\end{align*}
As $B>1$, it follows that
\begin{equation}\label{eq-main-size}
 \begin{aligned}
 \ln N
 &\leq\ln2+4m\ln B \in O(m^2\log m+mn\log R)\\
 &=O(n^{2h}\log n+n^{h+1}\log R).
 \end{aligned}
\end{equation}
Consequently, there is a constant $C=C(h,q)>0$ such that
\[
 R_{\mathrm{edge}}(F^<;q)\leq N
 \leq\left(2^{m^2\ln(m+1)}R_h(F;q)^{mn}\right)^C,
\]
which proves the first bound when $F^<$ has no isolated vertices.
Moreover, $n\leq hm$ and $m+1\leq2n^h\leq R^{h+1}$ imply $\ln N \in O(m^2\log R)$ by~\eqref{eq-main-size}, so increasing $C$ also gives $R_{\mathrm{edge}}(F^<;q)\leq R_h(F;q)^{C m^2}$.

For the remaining bounds in Theorem~\ref{thm-main}, recall the classical estimate
$R\leq R_h(n;q)\leq\tow_h(bn)$ for some constant $b=b(h,q)\geq1$~\cite{erdRado52,ramseyBook}.
If $h=2$, then $\ln R\in O(n)$, so \eqref{eq-main-size} gives
\[
 \ln N \in O(n^4\log n+n^4)=O(n^4\log n).
\]
Thus $N\leq2^{A n^4\ln n}$ for a suitable constant $A=A(2,q)$.

Now assume that $h\geq3$.
By \eqref{eq-main-size} and the classical estimate, there is a constant $D=D(h,q)\geq1$ such that
\[
 \log N\leq D n^{2h+1}\tow_{h-1}(bn).
\]
For $k\geq2$, $x\geq1$, and every nonnegative integer $t$, the definition of the tower function gives
$\tow_k(x+t)\geq2^t\tow_k(x)$.
Applying this with $k=h-1$ and
$t=\lceil\log(Dn^{2h+1})\rceil\in O(\log n)$, we obtain
\[
 \log N\leq\tow_{h-1}(bn+t)
 \leq\tow_{h-1}(An)
\]
for a sufficiently large constant $A=A(h,q)$.
Thus, $N\leq\tow_h(An)$.

Finally, assume that $F^<$ has $z>0$ isolated vertices, and let $F_0^<$ be obtained by deleting them.
Adjoining $z$ isolated vertices to a Ramsey host for $F_0^<$ gives
\[
 R_{\mathrm{edge}}(F^<;q)\leq R_{\mathrm{edge}}(F_0^<;q)+z.
\]
Since $R_h(F_0;q)\leq R_h(F;q)$, $|V(F_0)|\leq n$, and $z\leq n\leq R_h(F;q)$, the first two bounds extend to $F^<$ by increasing $C$.
The bounds in terms of $n$ similarly extend by increasing $A$.
\end{proof}

\paragraph{Remark.}
If $F^<$ in Theorem~\ref{thm-main} is complete, the same bounds hold when the monochromatic copy is additionally required to preserve an arbitrary prescribed vertex order.
For $q=1$ or $n\leq h$ this is immediate.
Otherwise, we label the vertices of $F^<$ increasingly by $[n]$ in that order and give the host $[N]$ its natural order.
In the construction above, restrict $\phi$ to the increasing bijection onto each $n$-set $U$ and use $a=1$ for this restricted family.
Since $F$ is complete, $R=R_h(n;q)$ still guarantees the copies used in Lemma~\ref{lem-missing}, and the remaining argument is unchanged.

\paragraph*{Use of AI-assisted technologies.}
The authors used ChatGPT-6 Astra to edit the paper to improve readability and to assist with the calculations in the proofs.
The authors independently checked all resulting mathematical arguments, rewrote the proofs to improve their clarity and readability, and take full responsibility for the final content.

%\bibliography{bibliography}
%\bibliographystyle{plainnat}
%\include{bibliography}
\printbibliography[heading=bibintoc]

\end{document}